\documentclass[a4paper, 12pt, reqno]{amsart}

\usepackage[a4paper]{geometry}
\usepackage[utf8]{inputenc}
\usepackage[T1]{fontenc}
\usepackage[english]{babel}
\usepackage{bm}
\usepackage{amssymb}
\usepackage{amsthm}
\usepackage{amsmath}
\usepackage{hyperref}
\usepackage{mathtools}
\usepackage{cite}
\usepackage{pdfcomment}
\usepackage{tikz-cd}
\usetikzlibrary{decorations.pathmorphing}
\usepackage{amsaddr}
\usepackage[pagewise]{lineno}

\allowdisplaybreaks

\newtheorem{theorem}{Theorem}[section]
\newtheorem{lemma}[theorem]{Lemma}
\newtheorem{proposition}[theorem]{Proposition}
\newtheorem{corollary}[theorem]{Corollary}
\newtheorem{example}[theorem]{Example}

\numberwithin{equation}{section}

\newcommand{\modulo}[3]{#1\equiv#2\textrm{ }(\textrm{mod }#3)}

\DeclarePairedDelimiter\langr{\langle}{\rangle}

\newcommand{\sm}[4]{\left(\begin{smallmatrix}#1&#2\\ #3&#4 \end{smallmatrix}
\right)}

\newcommand{\zNz}{\Z\slash N\Z}

\newcommand{\calH}{\mathcal{H}}
\newcommand{\calO}{\mathcal{O}}
\newcommand{\SL}{{\text {\rm SL}}}

\newcommand{\G}{\Gamma}
\newcommand{\tildeG}{\widetilde{\Gamma}}

\newcommand{\Q}{\mathbb{Q}}
\newcommand{\Z}{\mathbb{Z}}
\newcommand{\N}{\mathbb{N}}
\newcommand{\C}{\mathbb{C}}

\makeatletter
\renewcommand{\tocsection}[3]{%
	\indentlabel{\@ifnotempty{#2}{\bfseries\ignorespaces#1 #2\quad}}\bfseries#3}
	
\renewcommand{\tocsubsection}[3]{%
	\indentlabel{\@ifnotempty{#2}{\ignorespaces#1 #2\quad}}#3}

\newcommand\@dotsep{4.5}
\def\@tocline#1#2#3#4#5#6#7{\relax
	\ifnum #1>\c@tocdepth 
	\else
	\par \addpenalty\@secpenalty\addvspace{#2}%
	\begingroup \hyphenpenalty\@M
	\@ifempty{#4}{%
		\@tempdima\csname r@tocindent\number#1\endcsname\relax
	}{%
		\@tempdima#4\relax
	}%
	\parindent\z@ \leftskip#3\relax \advance\leftskip\@tempdima\relax
	\rightskip\@pnumwidth plus1em \parfillskip-\@pnumwidth
	#5\leavevmode\hskip-\@tempdima{#6}\nobreak
	\leaders\hbox{$\m@th\mkern \@dotsep mu\hbox{.}\mkern \@dotsep mu$}\hfill
	\nobreak
	\hbox to\@pnumwidth{\@tocpagenum{\ifnum#1=1\bfseries\fi#7}}\par
	\nobreak
	\endgroup
	\fi}
\AtBeginDocument{%
	\expandafter\renewcommand\csname r@tocindent0\endcsname{0pt}
}
\def\l@subsection{\@tocline{2}{0pt}{2.5pc}{5pc}{}}
\def\l@subsubsection{\@tocline{2}{0pt}{4.5pc}{5pc}{}}
\makeatother

\makeatletter
\def\subsubsection{\@startsection{subsubsection}{3}%
  \z@{.5\linespacing\@plus.7\linespacing}{-.5em}%
  {\normalfont\bfseries}}
\makeatother



\author{\tiny{Valentina Pribani\'{c}}}
\address{\tiny{Department of Mathematics, University of Zagreb\\
	 Bijeni\v{c}ka cesta 30, 10000 Zagreb, Croatia}}
\email{valentina.pribanic@gmail.com}

\title[Radical isogenies and modular curves]{Radical isogenies and modular curves}

\begin{document}
\begin{abstract}
This article explores the connection between radical isogenies and modular curves. Radical isogenies are formulas designed for the computation of chains of isogenies of fixed small degree $N$, introduced by Castryck, Decru, and Vercauteren at Asiacrypt 2020. One significant advantage of radical isogeny formulas over other formulas with a similar purpose is that they eliminate the need to generate a point of order $N$ that generates the kernel of the isogeny. While radical isogeny formulas were originally developed using elliptic curves in Tate normal form, Onuki and Moriya have proposed radical isogeny formulas of degrees $3$ and $4$ on Montgomery curves and attempted to obtain a simpler form of radical isogenies using enhanced elliptic and modular curves. In this article, we translate the original setup of radical isogenies in Tate normal form into the language of modular curves. Additionally, we solve an open problem introduced by Onuki and Moriya regarding radical isogeny formulas on $X_0(N).$
\end{abstract}

\maketitle
\tableofcontents

\section{Introduction}
Post-quantum cryptography (PQC) is an area of cryptography focused on developing cryptosystems that can resist attacks from both classical and quantum computers. These systems rely on hard mathematical problems that differ from the integer factorization problem or (elliptic-curve) discrete logarithm problem, which are the basis of most current cryptographic algorithms. PQC  includes various approaches to cryptography, such as lattice-based cryptography, code-based cryptography, multivariate-based cryptography, hash-based cryptography, and isogeny-based cryptography. 

The first isogeny-based cryptosystem was proposed by Couveignes in 1997 \cite{couveignes2006hard}, and then again independently by Rostovtsev and Stolbunov in 2006 (commonly referred to as CRS)  \cite{rostovtsev2006public}. They described a non-interactive key exchange using ordinary elliptic curves. New momentum in this field came in 2011 when De Feo and Jao proposed SIDH \cite{jao2011towards}, the supersingular isogeny Diffie-Hellman key exchange. A variant of this algorithm called SIKE was a promising candidate for NIST PQC standardization,\footnote{More information about standardization is available as \url{https://csrc.nist.gov/Projects/post-quantum-cryptography/post-quantum-cryptography-standardization}.} but it was broken in several independent papers in August 2022 \cite{castryck2022efficient,maino2022attack,robert2022breaking}. In 2018, Castryck, Lange, Martindale, Panny, and Renes introduced CSIDH \cite{castryck2018csidh}, or commutative-SIDH, a key exchange protocol that adapts CRS protocol to supersingular elliptic curves. CSIDH is not affected by the previously mentioned attacks.  

Compared to other post-quantum protocols, the main advantages of isogeny-based cryptography are smaller key sizes and ciphertext sizes. On the other hand, the main disadvantage of isogeny-based protocols has been the high computational cost of encryption and decryption. These advantages and disadvantages are particularly evident in digital signatures. SQISign, introduced in 2020 \cite{de2020sqisign}, is among the most promising and compact isogeny-based digital signatures. It has seen some speed improvements in 2022 \cite{de2022new}, but despite this, it is still several orders of magnitude slower than other post-quantum signature schemes.

Protocols like CRS, CSIDH or, for example, Charles, Goren and Lauter's hash function \cite{charles2009cryptographic} share the need to compute isogenies of low degree in finite field. An isogeny can be computed from the coordinates of the points in its kernel using V\'{e}lu's formulas \cite{velu1971isogenies}. To improve and accelerate isogeny computation, various approaches and variants of V\'{e}lu's formulas have been proposed for different curve models, such as Montgomery curves in \cite{costello2017simple}, Edwards curves \cite{cervantes2019stronger,kim2019optimized}, and Hessian curves \cite{broon2021isogenies}. An algorithm by Bernstein, De Feo, Leroux and Smith \cite{bernstein2020faster} reduces the cost of computation of isogeny of degree $N$ from $\calO(N)$ to $\widetilde{\calO}(\sqrt{N})$ and can be applied to Huff's and general Huff's curves \cite{cryptoeprint:2021/073}.

Radical isogenies are formulas designed for computing a chain of isogenies of the same small degree between elliptic curves over a finite field. They were first introduced by Castryck, Decru and Vercauteren in 2020 \cite{castryck2020radical}. The authors showed that using radical isogeny formulas in CSIDH-512 leads to a more efficient implementation and a speed-up of $19\%,$ see \cite[Section~6]{castryck2020radical}. In \cite{castryck2020radical}, formulas were given for $N\leq 13,$ and in 2022, the same group of authors, along with Houben \cite{castryck2023horizontal}, developed a different method for finding radical isogeny formulas for a given degree $N$, and provided formulas for $N\leq 37$.

The concept of radical isogeny formulas was initially introduced for elliptic curves in Tate normal form. Generally, an elliptic curve over a field $k$ and a point on that curve with an order of at least $N\geq 4$ are isomorphic to an elliptic curve of the form $E\colon y^2+(1-c)xy-by = x^3 - bx^2$ with $b,c\in k$, and a point $P=(0,0)$ of the same order $N$. This form is known as the Tate normal form and it provides two coefficients, denoted $b$ and $c$. Given a cyclic isogeny $\varphi\colon E\xrightarrow{} E'=E\slash\langr{P}$, radical isogeny formulas compute points $P'$ of order $N$ on $E'$ such that composition $E\xrightarrow{\varphi}E'\xrightarrow{}E'\slash\langr{P'}$ is cyclic of degree $N^2$. The coordinates of $P'$ are elements of the smallest field that contains the coefficients $b$ and $c$, along with a radicand $\rho$ that is a $N$-th root of a rational expression in the coefficients $b$ and $c$. The elliptic curve $E'$ and point $P'$ are also isomorphic to an elliptic curve in Tate normal form (for example, defined with coefficients $b'$ and $c'$) and a point $(0,0)$ of order $N$. This allows us to use radical isogeny formulas again, making the process iterative. The coefficients $b'$ and $c'$ can be expressed as elements of the same field as $P'$.

As a first contribution of this article, in Section \ref{sec:RadicalToModular}, we will extend the notion of radical isogeny formulas to the language of modular curves. To achieve this, we will utilize enhanced elliptic curves, which are curves paired with additional torsion data and affiliated with some congruence subgroup. The aforementioned parameters from Tate normal form and the radicand $\rho$ can all be regarded as functions on the set of equivalence classes of enhanced elliptic curves. This generalization of radical isogenies for degree $N$ is directly related to the modular curve $X_1(N)$, congruence subgroup $\G_1(N)$ and pairs of enhanced elliptic curves consisting of an elliptic curve and a point of order $N$.

In \cite{onuki2022radical}, Onuki and Moriya introduced radical isogeny formulas of degrees $3$ and $4$ on Montgomery curves. A Montgomery curve over a field $k$ is an elliptic curve of the form $E\colon y^2=x^3+Ax^2+x,$ where $A\in k$ and $A^2\neq 4$. The coefficient $A$ is called the Montgomery coefficient of $E$. For degree $4$ (degree $3$ is similar), the set of equivalence classes of enhanced elliptic curves for $\G_0(4)$, denoted by $S_0(4)$, is equal to the set of equivalence classes of enhanced elliptic curves for $\G_1(4)$. This equality implies the existence of radical isogenies formulas for the modular curve $X_0(4)$. The Montgomery coefficient $A$ represents a class in the set $S_0(4)$, see \cite[Section~2.3]{onuki2022radical}. In other words, we can say that the coefficient $A$ describes an enhanced elliptic curve where the additional torsion data is a cyclic subgroup of order $4$. The Montgomery coefficient for the curve $E'$ can be calculated by a rational expression depending on the fourth root from $4(A+2)$ see \cite[Theorem~8]{onuki2022radical}. 

The authors of \cite{onuki2022radical} explored the possibility of extending radical isogeny formulas to the modular curve $X_0(N)$ when $N\geq 5$. The idea behind this can be summarized in a few informal steps. First, take a modular curve of genus zero, such as $X_0(5)$. Then, find a parameter that specifies its set of equivalence classes of enhanced elliptic curves, find a model of a universal elliptic curve for $X_0(5)$ defined by that parameter (Tate, Montgomery, or something else) and then find a radical isogeny formula on such a curve. This approach is presented as an example, see \cite[Section~4]{onuki2022radical} and Section \ref{sec:OpenProblem}, that argues against the existence of radical isogeny formulas for that curve. While this example indicates that finding radical isogenies for degrees greater than $4$ is maybe not possible, a general answer was left as an open problem. This article provides a solution to that open problem, i.e. in Corollary \ref{cor:OpenProblem} we prove that radical isogeny formulas cannot exist on the set of equivalence classes of enhanced elliptic curves for $\G_0(N)$ when $N\geq 5$.

\section*{Paper organization}
Section \ref{sec:Preliminaries} provides necessary background, including brief overview on elliptic curves, isogenies of elliptic curves, the definition of congruence subgroups, modular curves, semidirect product of groups, radical isogenies and the description of the previously mentioned open problem in Example \ref{ex:OpenProblem}. In the section \ref{sec:RadicalToModular} we generalize radical isogenies using modular curves. Section \ref{sec:RadicalToModularX0} extends the setting from Section \ref{sec:RadicalToModular} to include modular curve $X_0(N)$. In the same section Theorem \ref{thm:MainTheorem} is proved, and a corollary of that theorem is a solution to the open problem from Example \ref{ex:OpenProblem}.

\section{Preliminaries}\label{sec:Preliminaries}
This section will provide summary of necessary background. For more details on elliptic and modular curves refer to \cite{silverman2009arithmetic}, \cite{diamond2005first} and \cite[Chapter~III]{cornell2013modular}.

\subsection{Elliptic curves}
Let $k$ be a field.	 An elliptic curve $E$ over $k$ is a smooth projective curve of genus one with a specified base point $\calO_E$. Group of all the points on $E$ defined over $k$ is denoted by $E(k)$. Given an integer $N$, multiplication by $N$ map is denoted with $[N]$. The kernel of this map is the $N$ torsion subgroup, $E[N]=\{P\in E(\overline{k})\colon[N]P=\calO_E\}$. A point $P$ on the curve $E$ is of order $N$ if $[N]P=\calO_E$ and $[m]P\neq \calO_E$ for $m<N$. For a curve $E$ as above and a point $P$ of order $N\geq 4,$ the following Lemma holds:

\begin{lemma}\label{lemma:TateNormalForm}
Let $E$ be an elliptic curve over $k$ and let $P\in E(k)$ be a point of order $N\geq 4$, 		then the pair $(E,P)$ is isomorphic to a unique pair of the form
\begin{equation}\label{eqn:TateNormalFormDef}
E\colon y^2+(1-c)xy-by = x^3 - bx^2,\ P=(0,0)
\end{equation}
with $b, c\in k$ and $$\Delta(b,c)=b^3(c^4-8bc^2-3c^3+16b^2-20bc+3c^2+b-c)\neq 0.$$
\end{lemma}
\noindent Curve $E$ in (\ref{eqn:TateNormalFormDef}) is said to be in Tate normal form. For proof see \cite[Lemma~2.1]{streng2015generators}.

If $\text{char}(k)\nmid N$, we can define the Tate pairing as a bilinear map
$$t_N\colon E(k)[N]\times E(k)\slash NE(k)\xrightarrow{} k^*\slash(k^*)^N\colon (P_1,P_2)\mapsto t_N(P_1,P_2),$$ where $E(k)[N]$ consists of all the points in $E[N]$ defined over $k$.

Following \cite[Chapter~II.3]{silverman2009arithmetic}, a divisor for a curve $E$ is defined as a formal sum $\sum_{P\in E}n_P(P),$ where $n_P\in\Z$ and $n_P=0$ for all but finitely many $P\in E.$ A Miller function $f_{N,P_1}$ is any function on $E$ with divisor $N(P_1)- N(\calO_E).$ 
The support of a divisor is the set of points $P\in E$ for which $n_P\neq 0.$ 
Let $D$ be a $k$-rational divisor on $E$ that is linearly equivalent to $(P_2)-(O_E)$ and whose support is disjoint from $\{P_1,\calO_E\}$. The support of this divisor is disjoint from the divisor of Miller function $f_{N,P_1},$ thus $f_{N,P_1}(D)=\prod_{P\in E}f_{N,P_1}(P)^{n_P}$ is well-defined. Then, the Tate pairing can be calculated as $t_N(P_1,P_2)=f_{N,P_1}(D).$ Furthermore, if $P_1\neq P_2$ and the Miller function is normalized, the Tate pairing $t_N(P_1,P_2)$ is equal to $f_{N,P_1}(P_2)$. When $f_{N,P}$ is a Miller function as above and $P$ point of order $N$, there exists a function $g_{N,P}\in \overline{k}(E)$ such that
\begin{equation}\label{eqn:MillerFGDef}
f_{N,P}\circ [N]=g_{N,P}^N.
\end{equation}
The function $g_{N,P}$ can be used to define the Weil pairing, see \cite[Chapter~III.8]{silverman2009arithmetic} for details.

\subsection{Isogenies of elliptic curves}\label{sec:Isogenies}
Let $E$ and $E'$ be elliptic curves over $k$. An isogeny $\varphi\colon E\xrightarrow{} E'$ is a non-constant morphism satisfying $\varphi(\calO_E)=\calO_{E'}$. An example of an isogeny is multiplication by $N$. Except for the zero isogeny, every other isogeny is a finite map of curves, so there is a usual injection of function fields $\varphi^{*}\colon \overline{k}(E')\xrightarrow{}\overline{k}(E).$ The degree of $\varphi$, denoted by $\deg(\varphi),$ is the degree of the finite extension $\overline{k}(E)\slash\varphi^{*}(\overline{k}(E')).$ An isogeny is separable (inseparable, purely inseparable) if this finite extension is separable (inseparable, purely inseparable). There exists a dual isogeny $\widehat{\varphi}\colon E'\xrightarrow{} E$ for every isogeny $\varphi$. This dual isogeny satisfies $\widehat{\varphi}\circ\varphi=[\deg(\varphi)]$. A kernel of an isogeny is a finite subgroup of $E(\overline{k})$. The size of the kernel divides the degree of the isogeny, and they are equal when the isogeny is separable. Given a finite subgroup $C\subset E$ there exists a unique separable isogeny having domain $E$, codomain $E\slash\langr{C}$, and $C$ as its kernel. V\'{e}lu's formulas can be used to calculate this isogeny, see \cite[Theorem~1]{castryck2020radical} for a complete list of formulas.

\subsection{Congruence subgroups, modular and enhanced elliptic curves}
The group of $2\times 2$ matrices with integer entries and determinant equal to $1$ is 
$$\SL_2(\Z)=\{\sm{a}{b}{c}{d}\colon a,b,c,d\in\Z, ad-bc=1\}.$$
The principle congruence subgroup for $N>0$ is defined as
$$\G(N)=\{\sm{a}{b}{c}{d}\in\SL_2(\Z)\colon\modulo{\sm{a}{b}{c}{d}}{\sm{1}{0}{0}{1}}{N}\}.$$ 
The reduction modulo $N$ morphism $\Z\xrightarrow{}\zNz$ induces a homomorphism \mbox{$\SL_2(\Z)\xrightarrow{}\SL_2(\zNz)$} with kernel $\G(N)$, thus $\G(N)$ is normal subgroup in $\SL_2(\Z)$ of finite index. This homomorphism is a surjection, so there is an induced isomorphism $$\SL_2(\Z)\slash\G(N)\xrightarrow{\sim}\SL_2(\zNz).$$ 
Other standard congruence subgroups are 
\begin{align*}
\G_1(N)&=\{\sm{a}{b}{c}{d}\in\SL_2(\Z)\colon\modulo{\sm{a}{b}{c}{d}}{\sm{1}{*}{0}{1}}{N}\},\\
\G_0(N)&=\{\sm{a}{b}{c}{d}\in\SL_2(\Z)\colon\modulo{\sm{a}{b}{c}{d}}{\sm{*}{*}{0}{*}}{N}\}.
\end{align*}
These subgroups satisfy $\G(N)\subset\G_1(N)\subset\G_0(N)\subset\SL_2(\Z).$

Let $\calH=\{\tau\in\C\colon\text{Im}(\tau)>0\}$ be the upper half-plane and let $\sm{a}{b}{c}{d}$ in $\SL_2(\Z)$ be a matrix. The action of the matrix on $z\in\calH$ is defined by 
$$\sm{a}{b}{c}{d}(z)=\frac{az+b}{cz+d}.$$
Using this fractional linear transformation, for a congruence subgroup $\G$, we can define the modular curve by $$Y(\G)=\G\slash\calH=\{\G\tau\colon\tau\in\calH\}.$$ 
For $\G(N),\G_1(N),\G_0(N),$ 
$$Y(N)=\G(N)\slash\calH, Y_1(N)=\G_1(N)\slash\calH\ \text{and}\ Y_0(N)=\G_0(N)\slash\calH.$$
If the action is extended to $\calH^{*}=\calH\cup\Q\cup\{\infty\},$ following modular curves can be defined $$X(\G)=\G\slash\calH^{*}, X(N)=\G(N)\slash\calH^{*}, X_1(N)=\G_1(N)\slash\calH^{*}\ \text{and}\ X_0(N)=\G_0(N)\slash\calH^{*}.$$ 
Let $E$ be an elliptic curve over algebraically closed field whose characteristic does not divide $N.$ An enhanced elliptic curve for $\G_0(N)$ is an ordered pair $(E,C),$ where $C$ is a cyclic subgroup of $E$ of order $N$. Two enhanced elliptic curves $(E,C)$ and $(E',C')$ are equivalent if there exists an isomorphism $E\xrightarrow[]{\sim} E'$ that takes $C$ to $C'$. We denote the set of equivalence classes of enhanced elliptic curves for $\G_0(N)$ by $$S_0(N)=\{\text{enhanced elliptic curves for}\ \G_0(N)\}\slash\sim.$$ Similarly, an enhanced elliptic curve for $\G_1(N)$ is a pair $(E,P),$ where $P$ is a point of order $N$. Two enhanced elliptic curves $(E,P), (E',P')$ are equivalent if there exists an isomorphism $E\xrightarrow[]{\sim} E'$ that takes $P$ to $P'$. We denote the set of equivalence classes of enhanced elliptic curves for $\G_1(N)$ by $$S_1(N)=\{\text{enhanced elliptic curves for}\ \G_1(N)\}\slash\sim.$$ 

Following \cite[Chapter~1.3]{diamond2005first}, we can define the complex elliptic curve $E_{\tau}$  as the quotient of the complex plane by the lattice $$E_{\tau}\coloneqq\C\slash\Lambda_{\tau}=\{z+\Lambda_{\tau}\colon z\in\C\},$$ where $\Lambda_{\tau}=\Z\oplus\tau\Z$. Definition of the sets  $S_0(N)$ and $S_1(N)$ from the previous paragraph remains unchanged when the underlying field is $\C$ and $E$ is a complex elliptic curve. Points of $Y_1(N)$ are in bijection with isomorphism classes of pairs $(E,P)\in S_1(N).$ To establish this bijection, to $\tau\in\calH,$ associate the pair $(E_{\tau}, \frac{1}{N}+\Lambda_{\tau}).$ Any pair $(E,P)$ is isomorphic to $(E_{\tau}, \frac{1}{N}+\Lambda_{\tau})$ for some $\tau\in\calH$ and $E_{\tau}$ is isomorphic to $E_{\tau'}$ if and only if $\tau'\in\G_1(N)\tau.$ We have the following theorem.

\begin{theorem}\label{thm:moduliSpaceS1}
Let $N$ be a positive integer. The moduli space for $\G_1(N)$ is $$S_1(N)=\{[E_{\tau}, \frac{1}{N}+\Lambda_{\tau}]\colon\tau\in\calH\}.$$ Two points $[E_{\tau}, \frac{1}{N}+\Lambda_{\tau}]$ and $[E_{\tau'}, \frac{1}{N}+\Lambda_{\tau'}]$ are equal if and only if $\G_1(N)\tau=\G_1(N)\tau'.$ Thus, there is a bijection $$\psi_1\colon S_1(N)\xrightarrow{\sim} Y_1(N),\quad [\C\slash\Lambda_{\tau},\frac{1}{N}+\Lambda_{\tau}]\mapsto \G_1(N)\tau.$$
\begin{proof}
See \cite[Theorem~1.5.1.]{diamond2005first}.
\end{proof}
\end{theorem}
\noindent Theorem \ref{thm:moduliSpaceS1} has analogous versions for congruence subgroups $\G_0(N)$ and $\G(N)$, also part of the \cite[Theorem~1.5.1.]{diamond2005first}.

\subsection{Semidirect product of groups}\label{sec:SemidirectProduct}
Following \cite{KonradSemidirect}, for two groups $G_1$ and $G_2$ and an action $\widehat{\varphi}\colon G_2\rightarrow \text{Aut}(G_1)$ of $G_2$ on $G_1$ (by automorphisms), the corresponding semidirect product $G_1\rtimes_{\widehat{\varphi}} G_2$ is defined as a set 
$$G_1\times G_2 = \{(g_1,g_2)\colon g_1\in G_1, g_2\in G_2\},$$ 
where the group law on $G_1\rtimes_{\widehat{\varphi}} G_2$ is 
$$(g_1,g_2)(g_1',g_2') = (g_1\widehat{\varphi}_{g_2}(g_1'),g_2g_2').$$
Element $(e_{G_1},e_{G_2})$ is the identity, and inverse for an element $(g_1,g_2)$ is $$(g_1,g_2)^{-1}=(\widehat{\varphi}_{g_2^{-1}}(g_1^{-1}),g_2^{-1})=((\widehat{\varphi}_{g_2^{-1}}(g_1))^{-1},g_2^{-1}).$$ Examples of subgroups are $G_1\times {e_{G_2}}=\{(g_1,e_{G_2})\colon g_1\in G_1\}$ which is a normal subgroup, and ${e_{G_1}} \times G_2=\{(e_{G_1},g_2)\colon g_2\in G_2\}$. 

\subsection{Radical isogenies}\label{sec:RadicalIsogenies}
Following \cite{castryck2020radical}, this section will provide a necessary background on radical isogenies. Let $k$ be a field, $N\geq 4$ such that $\text{char}(k)\nmid N$. Consider an elliptic curve $E$ over $k$ and a point $P\in E(k)$ of order $N$. Using Lemma \ref{lemma:TateNormalForm}, the curve-point pair $(E,P)$ is isomorphic to a unique pair of a curve $$y^2+(1-c)xy-by=x^3-bx^2,$$ where $b,c\in k,$ and a point $(0,0)$ of order $N$. There exists an isogeny $\varphi\colon E\xrightarrow{} E\slash\langr{P}$ with $\langr{P}$, a cyclic subgroup generated by the point $P$, as a kernel. We denote curve $E\slash\langr{P}$ over $k$ by $E'$ and let $P'$ be a point on $E'$ of order $N$ such that $\widehat{\varphi}(P')=P,$ where $\widehat{\varphi}$ is a dual isogeny of $\varphi.$ The point $P'$ satisfying this condition is called $P$-distinguished and it is not unique. According to \cite[Theorem~5]{castryck2020radical} the coordinates of the point $P'$ can be expressed using a formula that depends on $b,c$ and $\sqrt[N]{\rho},$ where $\rho$ is a representative of Tate pairing $t_N(P,-P)$. Hence, the point $P'$ is defined over $k(b,c,\sqrt[N]{\rho}).$ As $P'$ is of order $N$ on curve $E',$ a Tate normal form for this pair can be defined by the unique coefficients $b'$ and $c'.$ The iterative process of radical isogeny formulas can be repeated on pair $(E',P')$. Moreover, the formulas for $b'$ and $c'$ can be expressed directly as elements of the field extension $k(b,c,\sqrt[N]{\rho})$, which is a simple radical\footnote{A field extension $K\subset L$ is a simple radical extension of degree $N\geq 2$ if there exists an $\alpha$ such that $L=K(\alpha),\alpha^N\in K$, and $x^N-\alpha^N\in K[x]$ is irreducible.} extension of $k(b,c)$. The explicit radical isogeny formulas when $N=5,$ are written in the following example:

\begin{example}[\!\!\protect{\cite[Section~4]{castryck2020radical}}]
Let $N=5.$ Elliptic curve $E$ is of the form $$y^2+(1-b)xy-by=x^3-bx^2,$$ and, using V\'{e}lu's formulas, curve $E'$ is equal to $$y^2+(1-b)xy-by=x^3-bx^2-5b(b^2+2b-1)x-b(b^4+10b^3-5b^2+15b-1).$$ With some details omitted, $\rho=f_{5,P}(-P)=b,\alpha=\sqrt[5]{\rho}$ and point $P'$ has coordinates
\begin{align*}
x_0'&= 5\alpha^4+(b-3)\alpha^3+(b+2)\alpha^2+(2b-1)\alpha-2b, \\
y_0'&= 5\alpha^4+(b-3)\alpha^3+(b^2-10b+1)\alpha^2+(13b-b^2)\alpha-b^2-11b.
\end{align*}
After translating point $P'$ to $(0,0),$ isomorphic curve in Tate normal form will be $$E'\colon y^2+(1-b')xy-b'y=x^3-b'x^2,$$ where $$b'=\alpha\frac{\alpha^4+3\alpha^3+4\alpha^2+2\alpha+1}{\alpha^4-2\alpha^3+4\alpha^2-3\alpha+1}$$ and the process can be repeated. 
\end{example}
The standard method of calculating isogenies requires a point of a particular order for each isogeny in the chain. With radical isogeny formulas, such a point is only required for the initial step, i.e. the one step that uses V\'{e}lu's formulas. Subsequent steps can be calculated without any knowledge of torsion points. The list of formulas for radicand $\rho$ for $N\leq 13$ can be found in \cite[Section~5]{castryck2020radical} and link to a repository containing formulas for prime powers $16 < N \leq 37$ can be found in \cite[Section~4.3]{castryck2023horizontal}.

\subsubsection{Radical isogenies on Montgomery curves}\label{sec:OpenProblem}
In \cite{onuki2022radical}, Onuki and Moriya introduced radical isogeny formulas on Montgomery curves of degrees $3$ and $4.$ A Montgomery curve over a field $k$ is an elliptic curve of the form $$E\colon y^2=x^3+Ax^2+x,$$ where $A\in k$ and $A^2\neq 4.$ The coefficient $A$ determines a class of enhanced elliptic curve $(E,(0,0))$ in the set $S_0(4),$ see \cite[Section~2.3]{onuki2022radical} for details. Applying radical isogeny formulas on elements of set $S_1(N),$ i.e. on an enhanced elliptic curve $(E,P),$ results in a curve-point pair that is also an element of $S_1(N).$ When $N=3$ or $4$, the equality $S_0(N)=S_1(N)$ holds, and the existence of radical isogeny formulas on $S_1(3)$ and $S_1(4)$ implies a radical isogeny formula on $S_0(3)$ and $S_0(4),$ respectively. This means that there is a formula between Montgomery coefficients of curves, see \cite[Section~3]{onuki2022radical}. However, the methods used in \cite{onuki2022radical} for cases $N=3$ or $4$ cannot be directly applied to case $N\geq 5,$ partly because $S_0(N)\neq S_1(N).$ Moreover, developing radical isogeny formulas on $S_0(N)$ when $N\geq 5$ might not be possible, as illustrated by the following example.

\begin{example}[\!\!\protect{\cite[Section~4]{onuki2022radical}}]\label{ex:OpenProblem}
Let $N=5$. Let $k$ be a field with $\text{char}(k)\nmid N$, and $E,E'$ two elliptic curves over the field $k$ given in Tate normal form:
\begin{align*}
E &\colon y^2+(1-b)xy-by=x^3-bx, \\
E' &\colon y^2+(1-b')xy-b'y=x^3-b'x.
\end{align*}
Points $(0,0)$ are of order $5$ on these curves. The cyclic subgroup of $E$ generated by point $(0,0)$ is $$\{\calO_E,(0,0),(b,b^2),(b,0),(0,b)\}.$$
Pairs $(E,(0,0))$ and $(E',(0,0))$ are equivalent if and only if $b=b',$ while pairs $(E,\langr{(0,0)})$ and $(E',\langr{(0,0)})$ are equivalent if and only if $b=b'$ or $b=-\frac{1}{b'}.$ From this we have $\frac{b^2-1}{b}=\frac{b'^{2}-1}{b'},$ thus $\frac{b^2-1}{b}$ is a parametrization of $S_0(5).$ From radical isogeny formula we know that $b'$ is a rational expression in a fifth root of $b,$ i.e. $\Q(b')=\Q(\sqrt[5]{b}).$ Let $\beta=\frac{b^2-1}{b}$ and $\beta'=\frac{b'^{2}-1}{b'}.$ Field extension $\Q(b)\slash\Q(\beta)$ is of degree $2.$ Adjoining to the field extension $\Q(b')\slash\Q(\beta)$ a primitive fifth root of unity $\zeta_5\in\C,$ we obtain a Galois extension $\Q(\zeta_5)(b')\slash\Q(\zeta_5)(\beta)$ of degree $10$. Galois group of this extension $\text{Gal}(\Q(\zeta_5)(b')\slash\Q(\zeta_5)(\beta))$ is generated by automorphisms $\sigma\colon b'\mapsto-\frac{1}{b'}$ and $\tau\colon b'\mapsto\zeta_5b'$. The fixed field of $\sigma$ is $\Q(\zeta_5)(\beta')$, and of $\tau$ is $\Q(\zeta_5)(b).$ Because $\tau^{-1}\sigma\tau\neq\sigma,$ the group $\langr{\sigma}$ is not a normal subgroup of Galois group $\text{Gal}(\Q(\zeta_5)(b')\slash\Q(\zeta_5)(\beta))$, thus extension \mbox{$\Q(\zeta_5)(\beta')\slash\Q(\zeta_5)(\beta)$} cannot be a Galois extension. 
\end{example}

If the parameter $\beta'$ from Example \ref{ex:OpenProblem} could be expressed as a rational expression depending on the parameter $\beta,$ we would have a direct and simpler way (quadratic equation) to calculate $b'$, rather than the radical isogeny formulas. However, since the field extension \mbox{$\Q(\zeta_5)(\beta')\slash\Q(\zeta_5)(\beta)$} is not a Galois extension, this is not possible. Nevertheless, it may be possible to find a different $\beta',$ i.e. a different parametrization of $S_0(5)$ which will make the field extension $\Q(\zeta_5)(\beta')\slash\Q(\zeta_5)(\beta)$ Galois. 

\section{Radical isogenies in the language of modular curves}\label{sec:RadicalToModular}
Throughout this section we are using the same notation introduced in Section \ref{sec:RadicalIsogenies}, $E$ is the starting elliptic curve over a field $k$, $N\geq 4$ such that $\text{char}(k)\nmid N$, $P\in E(k)$ a point of order $N$, $E'$ a curve over $k$ defined with $E\slash\langr{P}$, $\varphi\colon E\xrightarrow{} E'$ an isogeny with kernel equal to $\langr{P}$ and $P'$ a point of order $N$ on $E'$ such that $\widehat{\varphi}(P')=P$. 

We will continue to work with enhanced elliptic curves for different congruence subgroups. For any elliptic curve $\widetilde{E}$ and point $\widetilde{P}$ of order $N\geq 4$, let its unique Tate normal form be defined with parameters $\widetilde{b}$ and $\widetilde{c}$. Let $\mathbf{b}$ denote a mapping $(\widetilde{E},\widetilde{P})\mapsto \widetilde{b},$ i.e. $\mathbf{b}$ is a function on the set of the enhanced elliptic curves for $\G_1(N),$ such that for a curve $(\widetilde{E},\widetilde{P})$ it returns parameter $\widetilde{b}$ from corresponding Tate normal form. This is a well-defined function because Tate's normal form is unique. Analogously, for parameter $\widetilde{c},$ function $\mathbf{c}\colon(E,P)\mapsto c$ is well-defined. Definition of modular functions on enhanced elliptic curves implies that $\mathbf{b}$ and $\mathbf{c}$ are elements of $k(X_1(N)).$ For curves $E$ and $E'$ we have $(E,P)\xmapsto{\mathbf{b}} b$, $(E,P)\xmapsto{\mathbf{c}} c$, $(E',P')\xmapsto{\mathbf{b}} b'$ and $(E',P')\xmapsto{\mathbf{c}} c'$. We would like to connect parameters $b,c$ with $b',c'$ using modular curves and maps on them. The following sequence of maps will be considered:
\begin{equation}\label{eqn:mapsX1}
\begin{aligned}
&(E,P)\xrightarrow{} (E',P')\xmapsto{\mathbf{b}} b', \\
&(E,P)\xrightarrow{} (E',P')\xmapsto{\mathbf{c}} c'.
\end{aligned} 
\end{equation}
Since the point $P'$ is not unique, the map $(E,P)\xrightarrow{} (E',P')$ is not uniquely defined, and therefore is no obvious connection on $X_1(N)$. For a point $P$ of order $N,$ let $R$ be a point on curve $E$ of order $N^2$ such that $[N]R=P$. This point $R$ is not unique. The pair $(E,R)$ is an enhanced elliptic curve for $\G_1(N^2).$ Let $P'$ be an image of a point $R$ under the isogeny $\varphi,$ i.e. $$P'\coloneqq \varphi(R) = R+\langr{P}.$$ This is a point of order $N$ on the curve $E'$. Since we have $$\hat{\varphi}(P')=\hat{\varphi}(\varphi(R))=[\deg\ \varphi]R=[N]R=P,$$ point $P'$ is $P$-distinguished. We can modify the sequence of maps in (\ref{eqn:mapsX1}) and continue to work with parameter $b$ and associated functions, as the approach for $c$ is the same. Beginning with the enhanced elliptic curve $(E,R),$ we have the following maps: 
\begin{align}
(E,R)\xrightarrow{} (E,[N]R) &= (E,P)\xmapsto{\mathbf{b}} b, \label{eqn:mapsX1Extb}\\
(E,R)\xrightarrow{} (E\slash\langr{[N]R},R+\langr{[N]R}) &= (E\slash\langr{P},R+\langr{P}) = (E',P')\xmapsto{\mathbf{b}} b' \label{eqn:mapsX1Extb'}.
\end{align} 
Using the mappings described in (\ref{eqn:mapsX1Extb'}), we can, similar to $\mathbf{b}$, define a function $\mathbf{b'}\colon (E,R)\mapsto b'$, which is a function on the set of enhanced elliptic curves for $\G_1(N^2).$ Maps and functions are visualized in Figure \ref{fig:mapsOnX1Connections}.

\begin{figure}[h]
\begin{center}
\begin{tikzcd}
	  & {X_1(N^2), \G_1(N^2)} \arrow[r, squiggly, no head] & {(E,R)} \arrow[d, "N\cdot" left] \arrow[rd, "\varphi"] & \\
	  & 											& {(E,[N]R)} \arrow[d, equal] 				   & {(E\slash\langr{[N]R},R+\langr{[N]R})} \arrow[d, equal]\\
      & {X_1(N), \G_1(N)} \arrow[r, squiggly, no head] & {(E,P)} \arrow[d, "{\mathbf{b}}"] 
      	   & {(E',P')} \arrow[d, "{\mathbf{b}}"] \\
      & 	& b  & b'                        
\end{tikzcd}
\caption{Maps on enhanced elliptic curves}
\label{fig:mapsOnX1Connections}
\end{center}
\end{figure}
The connection between parameters $b$ and $b'$ can now be extended to an enhanced elliptic curve $(E,R)$, i.e. to functions in $X_1(N^2)$. For every $N,$ let $\pi_{1,N}$ and $\pi_{2,N}$ define a pair of pullback operators:
\begin{align*}
\pi_{1,N}^{*} &\colon k(X_1(N))\xrightarrow{} k(X_1(N^2)),\ \pi_{1,N}((E,R))=(E,[N]R),\\
\pi_{2,N}^{*} &\colon k(X_1(N))\xrightarrow{} k(X_1(N^2)),\ \pi_{2,N}((E,R))=(E\slash \langr{[N]R},R+ \langr{[N]R}).
\end{align*}
From
$$(\pi_{1,N}^{*}\mathbf{b})(E,R)=\mathbf{b}(\pi_{1,N}(E,R))=\mathbf{b}(E,[N]R)=\mathbf{b}(E,P)$$
and
$$(\pi_{2,N}^{*}\mathbf{b})(E,R)=\mathbf{b}(\pi_{2,N}(E,R))=\mathbf{b}(E\slash\langr{[N]R}, R+\langr{[N]R})=\mathbf{b}(E',P')=\mathbf{b}'(E,R),$$ 
we can identify $\mathbf{b}$ and $\mathbf{b}'$ with their respective pullbacks by $\pi_{1,N}$ and $\pi_{2,N}$ and define $$b\coloneqq\pi_{1,N}^{*}\mathbf{b}\ \text{and}\ b'\coloneqq\pi_{2,N}^{*}\mathbf{b}$$ as functions on $X_1(N^2).$ Function $b'$ is an element of $\pi_{2,N}^{*}(k(X_1(N))),$ so if proved that there exist some modular function $g$ in $k(X_1(N^2)),$ defined using $b$ and $c,$ such that
\begin{equation}\label{eqn:todoFieldEqForN}
\pi_{1,N}^{*}(k(X_1(N)))(g)=\pi_{2,N}^{*}(k(X_1(N))),
\end{equation}
$b'$ will also be an element of $\pi_{1,N}^{*}(k(X_1(N)))(g).$

Let $P$ be a point of order $N$ as before, and let $f_{N,P}$ be a normalized Miller function. With the value of $f_{N,P}$ at point $-P,$ we can define a modular function $f$ on the set of enhanced elliptic curves for $\G_1(N)$ as: $$f\colon(E,P)\mapsto f_{N,P}(-P) \in k(X_1(N)).$$ For the function $f_{N,P}$ and the point $P$, from equation (\ref{eqn:MillerFGDef}), there exists a function $g_{N,P}\in \overline{k}(E)$ such that $f_{N,P}\circ [N]=g_{N,P}^N.$ 
Using this equality, for an enhanced elliptic curve $(E,R),$ where, as before, $P=[N]R,$ we have a function on $X_1(N^2)$ given by
\begin{align*}
(E,R)\mapsto f_{N,[N]R}(-[N]R) &= f_{N,[N]R}([N](-R)) \\
	&=g_{N,[N]R}(-R)^N=g_{N,P}(-R)^N.
\end{align*}
The function $g$ defined as $g\coloneqq(E,R)\mapsto g_{N,P}(-R)\in k(X_1(N^2))$ satisfies the property $$g^N=f,$$ which means that the $N$-th root of $f$ is a function on $X_1(N^2).$ Both functions $b,b'$, as well as function $g$ are elements of $k(X_1(N^2))$. However, due to the large size of this field, it is currently impossible to prove (\ref{eqn:todoFieldEqForN}). Thus, it is necessary to identify a smaller quotient of $X_1(N^2)$ where $b$, $b'$, and $g$ are well-defined.

\subsection{"Shrinking" the field of definition}
To gain a better understanding of the function $b'$, we will investigate the preimages of $(E,P)$ under the pullback operator $\pi_{2,N}.$ Specifically, we will investigate pairs $(E,R)$ and $(E,R')$ that are mapped by $\pi_{2,N}$ to the same point \mbox{$(E\slash\langr{[N]R}, R+\langr{[N]R})$.} For the equality
$$(E\slash\langr{[N]R'}, R'+\langr{[N]R'})=(E\slash\langr{[N]R}, R+\langr{[N]R})$$ to hold, we require $\langr{[N]R'}=\langr{[N]R}$ and $R'+\langr{[N]R'}=R+\langr{[N]R}.$ Combining these conditions, we get $R'+\langr{[N]R}=R+\langr{[N]R},$ which implies that there exists some $l\in\zNz$ such that $$R'=R+[l]\cdot([N]R)\ \text{and}\ [N]R'=[N](R+[l]P).$$
Therefore, we have $$\langr{[N](R+[l]P)} = \langr{[N]R}.$$
Since point $R$ has order $N^2$, the points $(E,R), (E, R+[1\cdot N]R),\dots,(E,R+[(N-1)\cdot N]R)$ are all mapped to the same final point. From the definition of $b'$, it is apparent that it is a function on $X_1(N^2)$ that maps points of this form to the same final point.

Let $t_m$ be an operator on $S_1(N^2)$ defined as $t_m\colon (E,\overline{P})\mapsto (E, [m]\overline{P}).$ When $m=N+1,$ define $t\coloneqq t_{N+1}.$ On an enhanced elliptic curve $(E,R)\in S_1(N^2),$ this operator act as follows:
$$(E,R)\xmapsto{t}(E,[N+1]R)\xmapsto{t}(E,[(N+1)^2]R)\xmapsto{t}\dots\xmapsto{t}(E,[(N+1)^{N-1}]R).$$ The order of the operator $t$ is equal to $N$ since we have $t^N(E,R)=(E,[(N+1)^N]R)=(E,R).$ Composing $t$ with $\pi_{1,N}$ on the enhanced elliptic curve $(E,R),$ we have:  
\begin{align*}
\pi_{1,N}(t(E,R)) 
		&= \pi_{1,N}(E,[N+1]R) \\
		&= (E,[N(N+1)]R)\ (\text{since order of}\ R\ \text{is}\ N^2) \\
		&= (E,[N]R) \\
		&= \pi_{1,N}(E,R),
\end{align*}
and for $\pi_{2,N}$:
\begin{align*}
\pi_{2,N}(t(E,R))
		&=\pi_{2,N}(E,[N+1]R) \\
		&=(E\slash\langr{[N(N+1)]R},[N+1]R+\langr{[N(N+1)]R}) \\
		&=(E\slash\langr{[N]R},R+\langr{[N]R})\ (\text{since}\ [N]R\in\langr{[N]R}) \\
		&= \pi_{2,N}(E,R),
\end{align*}
thus, every pullback by $\pi_{1,N}$ or by $\pi_{2,N}$ will be invariant under $t.$ Modular function $(E,R)\xmapsto{g} g_{N,P}(-R),$ with property $g^N=f,$ is also invariant under $t.$ Referring again to \cite[Chapter~III.8]{silverman2009arithmetic} for more details, function $g_{N,[N]R}$ can be used to define Weil pairing 
$$e_N(S,P)=\frac{g_{N,[N]R}(X+S)}{g_{N,[N]R}(X)},$$
where $X\in E$ and $S,P \in E[N]$ with $S=P$ allowed, and as before, we have $P=[N]R$. To see that function $g$ is invariant under $t,$ let $(E,R)\in S_1(N^2),$ then  
\begin{align*}
t(E,R)=(E,[N+1]R)&\xmapsto{g} g_{N,[N(N+1)]R}(-[N+1]R) \\
		  &= g_{N,[N]R}(-[N]R-R) \\
		  &= g_{N,[N]R}(-R-P)
\end{align*}
and together with the bilinearity and alternating property of Weil pairing,
\begin{align*}
g_{N,[N]R}(-R-P)&=g_{N,[N]R}(-R)e_{N}(-P,P)=g_{N,[N]R}(-R)e_N([N]P,P) \\
			  &=g_{N,[N]R}(-R)e_N(P,P)^{N-1}=g_{N,[N]R}(-R)=g_{N,P}(-R).
\end{align*}
   
Let $\langr{t}$ denote the group of automorphisms of $X_1(N^2)$ generated by $t$. A function on $X_1(N^2)$ that is invariant under the operator $t$ can be viewed as a function on the quotient $X_1(N^2)\slash\langr{t}.$ As discussed above, $b'$ is an example of such a function. The quotient $X_1(N^2)\slash\langr{t},$ i.e. the quotient of modular curve with the operator, is again a modular curve. To see this, following \cite{katz1985arithmetic} and \cite{deligne1973schemas} we can assume, for a field $k$ defined at the beginning of this section, that $k=\C.$ Then, we have the following proposition, which explicitly calculates the congruence subgroup defining this quotient, i.e. corresponding modular curve.

\begin{proposition}\label{prop:WellDefModularCurve}
Let $t$ be an operator defined on the set of enhanced elliptic curves for $\G_1(N^2)$ with $t(E,R)=(E,[N+1]R).$ Let $\langr{t}$ denote the subgroup of automorphisms of $X_1(N^2)$ generated by $t.$ The quotient of the extended upper half-plane $\calH^{*}=\calH\cup\Q\cup\{\infty\}$ and the congruence subgroup $$\tildeG(N)\coloneqq\left\{\sm{a}{b}{c}{d}\in \SL_2(\Z)\colon\modulo{c}{0}{N^2},\ \modulo{a,d}{1}{N}\right\},$$ i.e. $\tildeG(N)\slash\calH^{*},$ is a modular curve consisting of all the functions on $X_1(N^2)$ invariant under $t.$

\begin{proof}
As shown in \cite[Section~1.5]{diamond2005first}, sets of equivalence classes of enhanced elliptic curves can be used to describe the quotients of the upper half-plane by congruence subgroups. In other words, for a function $f$ on $X_1(N^2)\slash\langr{t},$ there is a corresponding meromorphic function $\mathbf{f}$ on the upper half-plane that is invariant under the action of $\G_1(N^2)$ and a matrix $\mathbf{t}\in\SL_2(\Z)$ corresponding to the operator $t.$ To see this, note that Theorem \ref{thm:moduliSpaceS1} shows that $S_1(N^2)$ is a moduli space of isomorphism classes of complex elliptic curves and $N^2$-torsion data, i.e. 
$$S_1(N^2)=\{[E_{\tau},\frac{1}{N^2}+\Lambda_{\tau}]\},$$
where $\tau,\Lambda_{\tau}$ and $E_{\tau}$ are defined as in Section \ref{sec:Preliminaries}. Describing what the operator $t$ does in the sense of congruence subgroup implies working with the pair $(E,R)$ after applying the operator $t,$ i.e. with
$$t(E_{\tau}, \frac{1}{N^2} + \Lambda_{\tau})=(E_{\tau}, \frac{N+1}{N^2} + \Lambda_{\tau}).$$ We need to find $\tau'\in\calH,$ such that $(E_{\tau}, \frac{N+1}{N^2} + \Lambda_{\tau})$ is isomorphic to $(E_{\tau'},\frac{1}{N^2}+\Lambda_{\tau'}).$ Let \mbox{$\tau'=\frac{(1-N)\tau-1}{N^2\tau+1+N}$} and $\Lambda_{\tau'}=\langr{1,\tau'}.$ Elements $1$ and $\tau$ are linear combination of $1$ and $\tau'$, which is obvious for $1$, and for $\tau$ we have: 
$$(1+N)(N^2\tau+N+1)\cdot\frac{(1-N)\tau-1}{N^2\tau+N+1}+(N^2\tau+N+1)\cdot 1=\tau.$$ 
From this, $\Lambda_{\tau'}$ is isomorphic to $\Lambda_{\tau}.$ Moreover, for the matrix $$\mathbf{t}=\sm{1-N}{-1}{N^2}{1+N}\in \G_0(N^2)\setminus \G_1(N^2),$$ using the usual fractional linear transformation on $\calH,$ we have $\mathbf{t}(\tau)=\tau'.$ Desired congruence subgroup $\tildeG(N)$ is generated by $\G_1(N^2)$ and matrix $\mathbf{t}$, thus 
$$\tildeG(N)=\left\{\sm{a}{b}{c}{d}\in \SL_2(\Z)\colon\modulo{c}{0}{N^2},\ \modulo{a,d}{1}{N}\right\}.$$ It is clear from the construction of the congruence subgroup $\tildeG(N)$ that the quotient $\tildeG(N)\slash\calH^{*}$ defines a modular curve consisting of all the functions on $X_1(N^2)$ invariant under $t.$
\end{proof}
\end{proposition}
As a direct consequence of Proposition \ref{prop:WellDefModularCurve}, $X_1(N^2)\slash\langr{t}$ is a well-defined modular curve with a function field equal to
$$k(X_1(N^2)\slash\langr{t})=\{f\in k(X_1(N^2))\colon f(t(E,R))=f(E,R),\forall (E,R)\in S_1(N^2)\}.$$ The following proposition shows the relationship between the congruence subgroups $\tildeG(N)$ and $\G_1(N^2)$.

\begin{proposition}\label{prop:NormalSubgroup}
Let $\tildeG(N)$ be a congruence subgroup defined as $$\tildeG(N)=\left\{\sm{a}{b}{c}{d}\in \SL_2(\Z)\colon\modulo{c}{0}{N^2},\ \modulo{a,d}{1}{N}\right\}.$$ The congruence subgroup $$\G_1(N^2)=\{\sm{a}{b}{c}{d}\in\SL_2(\Z)\colon\modulo{\sm{a}{b}{c}{d}}{\sm{1}{*}{0}{1}}{N^2}\}$$ is a normal subgroup of $\tildeG(N)$ with index $N$.
\end{proposition}

\begin{proof}
The congruence subgroup $\tildeG(N)$ is generated with the congruence subgroup $\G_1(N^2)$ and matrix $\mathbf{t}=\sm{1-N}{-1}{N^2}{1+N}\in \G_0(N^2)\setminus \G_1(N^2)$. To prove that $\G_1(N^2)$ is a normal subgroup of $\tildeG(N)$ it is enough to see that $\mathbf{t}^{-1}\sm{a}{b}{c}{d}\mathbf{t}\in\G_1(N^2),$ for every matrix $\sm{a}{b}{c}{d}\in\G_1(N^2)$. This is true because
\begin{align*}
\sm{1-N}{-1}{N^2}{1+N}&\sm{a}{b}{c}{d}\sm{1+N}{1}{-N^2}{1-N}= \\
&=\sm{a(1 - N)(N + 1) + c(1 - N) + N^2 (b (N + 1) + d)}
{-a (N + 1) +  b(N + 1)(N + 1) + d(N + 1) - c}
{c(1 - N)^2  - a N^2(1 - N) + N^2 (d (1 - N) - b N^2)}
{a N^2 +  d (1 - N)(N + 1) - b N^2(N + 1) - c (1 - N))}  \\
&\equiv\sm{1}{2(N+1)+b(2N+1)}{0}{1}{\ (\text{mod}\ N^2)}.
\end{align*}
To calculate the index of $\G_1(N^2)$ in $\tildeG(N)$ we will use the homomorphism \mbox{$\pi_N\colon\SL(\Z)\rightarrow\SL(\zNz)$,} induced by the reduction modulo $N$ for $N\geq 1$. The kernel of $\pi_N$ is the principal congruence subgroup $\G(N)$, which is a normal subgroup of finite index in $\SL(\Z)$. Any other congruence subgroup $\G(N)\subset\tildeG$ is of finite index in $\SL(\Z)$ and it is a preimage of $\pi_N$, i.e. $\tildeG=\pi_N^{-1}(\widehat{\G})$ where $\widehat{\G}$ is some subgroup of $\SL(\zNz).$ The index $[\tildeG\colon\G(N)]$ is equal to $\#\widehat{\G}.$ 

For $\#\widehat{\tildeG(N)}$, after reducing elements of $\tildeG(N)$ modulo $N^2$, the conditions on elements are $c=0, \modulo{a,d}{1}{N}$ and $a,b,c,d \in\Z\slash N^2\Z$. There are no conditions on $b,$ but $a$ and $d$ must satisfy a condition for determinant $\modulo{ad}{1}{N^2}.$ Writing $a=1+kN$ and $d=1+lN,$ where $k,l\in\{0,1,\dots,N-1\},$ we get
$$(1+kN)(1+lN)= \modulo{1+N(k+l)+klN^2}{1}{N^2},$$ which implies $\modulo{k+l}{0}{N},$ so $l$ depends completely on $k$. Therefore, $d$ depends completely on $a$. Altogether, $\#\widehat{\tildeG(N)}=N^3.$ The index $[\tildeG(N)\colon\G(N^2)]$ is equal to $[\tildeG(N)\colon\G_1(N^2)][\G_1(N^2)\colon\G(N^2)],$ thus
\begin{equation*}
[\tildeG(N)\colon\G_1(N^2)]=\frac{\#\widehat{\tildeG(N)}}{[\G_1(N^2)\colon\G(N^2)]}=\frac{\#\widehat{\tildeG(N)}}{N^2}=\frac{N^3}{N^2}=N.
\end{equation*}
\end{proof}

By performing a calculation similar to the one used in the proof of Proposition \ref{prop:NormalSubgroup}, it can be shown that the index of $[\G_1(N)\colon\G_1(N^2)]$ is equal to $N^2.$ Let $k(X_1(N^2))$ denote the function field corresponding to the modular curve $X_1(N^2)$. Using the results of the Proposition \ref{prop:NormalSubgroup}, the quotient $\tildeG(N)\slash\G_1(N^2)$ acts as a group of automorphism of $k(X_1(N^2))$ with fixed field $k(X(\tildeG(N))),$ i.e. 
$$k(X(\tildeG(N)))=k(X_1(N^2))^{\tildeG(N)\slash\G_1(N^2)}.$$
This gives us an equality of function fields: $$k(X(\tildeG(N))) = k(X_1(N^2))^{\mathbf{t}}.$$
So we have $$k(X(\tildeG(N)))=k(X_1(N^2)\slash\langr{t}).$$

We have shown that the function $b\in \pi_{1,N}^{*}(k(X_1(N)))$ is invariant under the operator $t.$ Therefore, $$\pi_{1,N}^{*}(k(X_1(N)))\ _\subset^N\ k(X(\tildeG(N)))=k(X_1(N^2)\slash\langr{t}),$$
where the degree of the extension is equal to the index of the subgroup. Returning to the equality (\ref{eqn:todoFieldEqForN}), the modular function $g\colon (E,R)\mapsto g_{N,P}(-R)$ is an element of the field $k(X_1(N^2)\slash\langr{t})$ with property $g^N=f$. The polynomial $x^N-f$ is a polynomial of degree $N$ in $\pi_{1,N}^{*}(k(X_1(N)))[x]$ having $g$ as a root. The equality (\ref{eqn:todoFieldEqForN}) depends on the irreducibility of the polynomial $x^N-f.$

\begin{lemma}\label{lemma:IrrPoly}
Let $f$ be a function defined on the set $S_1(N)$ with $(E,P)\mapsto f_{N,P}(-P),$ where $f_{N,P}$ is a normalized Miller function. Let $g$ be a function defined on the set $S_1(N^2)$ with $(E,R)\mapsto g_{N,P}(-R),$ where $P=[N]R$ and $f_{N,P}\circ [N]=g_{N,P}^N.$ Let $t\in Gal(k(X_1(N^2))\slash k(X_1(N)))$ be an operator defined as $t(E,R)=(E,[N+1]R),\ (E,R)\in S_1(N^2).$ Let $\pi_{1,N}^{*}\colon k(X_1(N))\xrightarrow{} k(X_1(N^2)$ be a pullback operator defined as $\pi_{1,N}((E,R))=(E,[N]R).$ Then, the polynomial $x^N-f$ is an irreducible polynomial in $\pi_{1,N}^{*}(k(X_1(N)))[x].$ 
\begin{proof}
We will show that the field extension $\pi_{1,N}^{*}(k(X_1(N)))(g)$ has degree $N$ over $k(X_1(N)),$ i.e. that the function $g$ is only invariant under the operator $t,$ thus it is an element of the function field $k(X(\tildeG(N))),$ and cannot be an element of some other field $k(X({\G}))$ with $\tildeG(N)\subsetneqq\G\subset\G_1(N)$ and $g\in k(X({\G})).$ 

Assume then that $g$ is invariant under another operator $T\in Gal(k(X_1(N^2))\slash k(X_1(N)))$ such that $T(E,R)=(E,T(R)),\ (E,R)\in S_1(N^2),$ and where $NT(R)=[N]R=P.$ The invariant property of the function $g$, together with the previously defined Weil pairing implies:
\begin{align*}
1=\frac{g_{N,[N]R}(-T(R))}{g_{N,[N]R}(-R)}=\frac{g_{N,P}((R-T(R))-R)}{g_{N,P}(-R)}=e_N(P,R-T(R)).
\end{align*}
The point $R-T(R)$ belongs to $E[N]$ because, by assuming $NT(R)=[N]R=P,$ we have $N(R-T(R))=P-P=\calO.$ Therefore, $e_N(P,R-T(R))$ is consistent with the definition of Weil pairing. From this, for every $(E,R)\in S_1(N^2),$ we have $e_N(P,R-T(R))=1.$ 

Let $E$ be a fixed elliptic curve and $P$ be a point of order $N$ on that curve such that $P=[N]R$. Since the Weil pairing is non-degenerate, and $e_N(P,R-T(R))=1$ for every $R$, it follows that the point $R-T(R)$ belongs to the subgroup $\langr{P}$. As a consequence, the point $T(R)$ can be written as $R+[l]P$ for some $l\in\Z$, which depends on $R$.

In comparison to the operator $t,$ since $g$ is invariant under $t,$ we have: 
$$g(E,R)=g(t(E,R))=g(E,[N+1]R)=g(E,R+P),$$ which implies $g(E,R)=g(E,R+[k]P)$ for every $k\in\Z.$ For the operator $T,$ we have: $$g(E,R)=g(E,T(R))=g(E,R+[l]P),$$ for some $l\in\Z.$ Therefore, the invariant property of the function $g$ under the operator $T$ follows from the invariant property of the function $g$ under the operator $t,$ which means that $g$ is modular only for the congruence subgroup $\tildeG(N).$ This implies that the function field $\pi_{1,N}^{*}(k(X_1(N)))(g)$ is an extension of degree exactly $N$ over $k(X_1(N)).$ The roots of the polynomial $x^N-f$ are of the form $\zeta_N^ng,$ where $\zeta_N$ represents the $N$-th root of unity and $n$ is a positive integer. If we assume that this polynomial is not irreducible, then we could find two non-constant polynomials $f_1,f_2\in k(X_1(N))[x],$ such that $x^N-f=f_1(x)f_2(x).$ However, this would lead to a contradiction since $g$ is a root for $f_1$ and has degree greater than or equal to $N$, which is the degree of $g$. Therefore, the polynomial $x^N-f$ is irreducible.

\end{proof}
\end{lemma}
In conclusion, the irreducibility of the polynomial $x^N-f,$ as stated in Lemma \ref{lemma:IrrPoly}, implies
$$\pi_{1,N}^{*}(k(X_1(N)))(g)=k(X_1(N^2)\slash\langr{t}),$$ which means $b'$ is an element of $\pi_{1,N}^{*}(k(X_1(N)))(g).$ Therefore, equality (\ref{eqn:todoFieldEqForN}) holds, and it is possible to generalize radical isogenies using modular functions.

\begin{example}
Let $N=5$ and $E$ be an elliptic curve over the field 
$$\Q_5(b,c):=\text{Frac}\ \frac{\Q[b,c]}{(F_5(b,c))}.$$
Tate normal form for $E$, together with the point $P$ of order $5$ is 
\begin{equation}\label{eqn:TateNormalFormDef5}
E\colon y^2+(1-b)xy-by=x^3-bx^2, \ P=(0,0).
\end{equation}
In general, polynomial $F_N(b,c)\in\Z[b,c]$ is an irreducible polynomial calculated from scalar multiples of the point $P.$ When $N\geq 4$, condition $F_N(b,c)=0$ together with $F_m(b,c)\neq 0,$ when $4\leq m<N$, and determinant of $E$ not equal to zero, ensures that the point $P$ is of order $N$. Other direction is also true, when $P$ is of order $N,$ then $F_N(b,c)=0.$ Additionally, $F_N(b,c)$ is a defining polynomial for the modular curve $X_1(N),$ so $\Q_N(b,c)$ is a function field of $X_1(N)$ over $\Q$. More details are available in \cite{streng2015generators}. 

In the case of $N=5$ we have $F_5(b,c)=b-c=0$, which implies a simpler Tate normal form (\ref{eqn:TateNormalFormDef5}). Having only a parameter $b$ results in only one modular function $\mathbf{b}$ in $k(X_1(5))$. On the other side, the curve $E'$ and the point $P'$ of order $5$ are given by $$E'=E\slash\langr{P}\colon y^2+(1-b')xy-b'y=x^3-b'x^2, \ P'=(0,0).$$ For a point $P$, let $R$ be a point of order $25$ such that $[5]R=P$. The pair $(E,R)$ is an enhanced elliptic curve for $\G_1(25).$ The pullbacks $\pi_{1,5},\pi_{2,5}$ and maps $b,b'$ are defined as before.

From the example in \cite[Section~4]{castryck2020radical}, when $N=5$, $f_{5,P}(-P)=b\in\Q_5(b).$ The fifth root of $b$ is a function on $X_1(25),$ as $(E,R)\xmapsto{g} g_{5,[5]R}(-R)$ is a well-defined map with a property $g^5=b$. 
 
Observing the preimages of $\pi_{2,5},$ points $(E,R), (E, R+[1\cdot 5]R), (E,R+[2\cdot 5]R),(E,R+[3\cdot 5]R)$ and $(E,R+[4\cdot 5]R)$ are all mapped to the same final point. The operator $t$ defined as $t(E,R)\mapsto(E,[5+1]R)=(E,[6]R)$ is of order $5$ and $\langr{t}$ is isomorphic to $\Z\slash 5\Z.$ The congruence subgroup generated by $\G_1(25)$ and matrix $\mathbf{t}=\sm{-4}{-1}{25}{6}$ is 
$$\tildeG(5)=\left\{\sm{\tilde{a}}{\tilde{b}}{\tilde{c}}{\tilde{d}}\in \SL_2(\Z)\colon\modulo{\tilde{c}}{0}{25},\ \modulo{\tilde{a},\tilde{d}}{1}{5}\right\}.$$  
Functions $b, b', g$ and every pullback by $\pi_{1,5}$ or $\pi_{2,5}$ are invariant under $t$, so they are also defined on the quotient $X_1(25)\slash\langr{t}$.
For the number of elements in group $\#\widehat{\tildeG(5)}$, after reducing elements of $\tildeG(5)$ modulo $25$, conditions on elements are $\tilde{c}=0,\modulo{\tilde{a},\tilde{d}}{1}{5}$ and $\tilde{a},\tilde{b},\tilde{c},\tilde{d} \in\Z\slash 25\Z$. The only possibilities for $\tilde{a}$ and $\tilde{d}$ are from the set $\{1,6,11,16,21\}.$ Since the determinant of the matrix has to be $1$ in $\SL(\Z\slash 25\Z),$ there are $25$ possibilities for $\tilde{b}$. Therefore, there are $125$ elements in this group, and the index $[\tildeG(5)\colon\G_1(25)]=5.$
The field extension $\pi_{1,5}^{*}(k(X_1(5)))\subset k(X_1(25)\slash\langr{t})$ has degree $5$, polynomial $X^5-b$ is irreducible in $\pi_{1,5}^{*}(k(X_1(5))),$ has a well-defined root, thus $$\pi_{1,5}^{*}(k(X_1(5)))(\sqrt[5]{b})=k(X_1(25)\slash\langr{t}),$$ meaning $b'\in \pi_{1,5}^{*}(k(X_1(5)))(\sqrt[5]{b})$ and $b'$ is a rational expression of $\sqrt[5]{b}.$
\end{example}

\section{Extending to $X_0(N)$}\label{sec:RadicalToModularX0}
Continuing from the setting of the previous section, the discussion for $\G_1(N)$, $X_1(N)$ and $S_1(N)$ can be extended to $\G_0(N),X_0(N)$ and $S_0(N)$. Let $\boldsymbol{\beta}$ be a function on enhanced elliptic curves for $\G_0(N)$, i.e. an element of $k(X_0(N)).$ For example, we can take $\boldsymbol{\beta}$ to be Hauptmodul\footnote{A Hauptmodul for a congruence subgroup $\G$ is a function that generates the field of modular functions for $\G$.} for $k(X_0(N)).$ Such Hauptmodul will exist if the genus of the modular curve is zero. Pullback operators $\pi_{1,N}$ and $\pi_{2,N}$ are defined as in the previous section, and $\psi_N$ is a pullback operator defined by
\begin{align*}
\psi_N^{*}\colon k(X_0(N))\xrightarrow{} k(X_1(N)),\ \psi_N((E,P))=(E,\langr{P}).
\end{align*}
Applying the compositions $\pi^{*}_{1,N}\circ\psi^{*}_N$ and $\pi^{*}_{2,N}\circ\psi^{*}_N$ to functions from $k(X_0(N))$ results in elements of $k(X_1(N^2)).$ From now on, we will identify the function $\boldsymbol{\beta}$ with $\beta\coloneqq\pi_{1,N}^{*}(\psi_N^{*}(\boldsymbol{\beta}))$ and define $\beta'\coloneqq\pi_{2,N}^{*}(\psi_N^{*}(\boldsymbol{\beta})).$ Both $\beta$ and $\beta'$ are elements of $k(X_1(N^2)).$ Maps and connections are visible in Figure \ref{fig:mapsX0Connections}.
\begin{figure}
\begin{tikzcd}
	  & {X_1(N^2), \G_1(N^2)} \arrow[r, squiggly, no head] 
	  			& {(E,R)} \arrow[d, "\pi_{1,N}" left] \arrow[rd, "\pi_{2,N}" right] & \\	
      & {X_1(N), \G_1(N)} \arrow[r, squiggly, no head] 
      			& {(E,[N]R)} \arrow[d,equal] \arrow[r] 	& {(E\slash{\langr{[N]R}},R+\langr{[N]R})} \arrow[d,equal] \\
      & {} 		& {(E,P)} \arrow[d, "\psi_N" left]	& {(E',P')} \arrow[d, "\psi_N" left] \\
      & {X_0(N), \G_0(N)} \arrow[r, squiggly, no head]
      			& {(E,\langr{P})} \arrow[d, "\boldsymbol\beta" left]	& {(E',\langr{P'})} \arrow[d, "\boldsymbol\beta" left] \\
      &   {}	& \beta(E,R)  & \beta'(E,R)                      
\end{tikzcd}
\caption{Maps on enhanced elliptic curves, including $X_0(N)$}
\label{fig:mapsX0Connections}
\end{figure}

Because $\beta'$ is defined as pullback by $\pi_2,$ as before, it is invariant under the operator $t$, which implies $\beta'\in k(X(\tildeG(N))).$ Similarly to the previous section, if radical isogeny formulas exist on $X_0(N)$ it should be possible to express $\beta'$ as an element of some function field depending on $\beta.$ To this end, we are interested in preimages of $(E,P),$ now under the maps $\pi^{*}_{1,N}(\psi^{*}_N)$ and $\pi^{*}_{2,N}(\psi^{*}_N),$ i.e. pairs of enhanced elliptic curves for $\G_1(N^2),(E,R)$ and $(E,R')$ mapped to the same final points $(E,\langr{[N]R})$ and $(E\slash\langr{[N]R},\langr{R+\langr{[N]R}}).$ Moreover, to include functions on $X_0(N)$, maps (\ref{eqn:mapsX1Extb}) and (\ref{eqn:mapsX1Extb'}) are extended to
\begin{equation}\label{eqn:mapsExtX0}
\begin{aligned}
(E,R)&\xrightarrow{}(E,[N]R)=(E,P)\xrightarrow{}(E,\langr{P}) \\
(E,R)&\xrightarrow{}(E\slash\langr{[N]R},R+\langr{[N]R})=(E',P')\xrightarrow{}(E',\langr{P'}).
\end{aligned}
\end{equation}
Describing preimages of maps in (\ref{eqn:mapsExtX0}) will result in another quotient of $X_1(N^2)$ where the function $\beta'$ will be well-defined. If we add enhanced elliptic curves for $\G_0(N)$ in maps (\ref{eqn:mapsExtX0}), i.e. maps $(E,P)\xrightarrow{}(E,\langr{P})$ and $(E',P')\xrightarrow{}(E',\langr{P'}),$ we obtain additional conditions on those preimages. Consequently, $\beta'$ will belong to a smaller function field $k(X(\G')),$ for some congruence subgroup $\G'$ satisfying $\tildeG(N)\subset\G'.$ The groups that describe the preimages and their connections to the function fields are provided in the following lemma.

\begin{lemma}\label{lemma:MainLemma}
Let $N\geq 5$ be a positive integer. Group $G,$ defined as a semidirect product 
$$G=(\zNz)^{2}\rtimes_{\widehat{\varphi}}(\zNz)^{\times},$$ 
where for a triple $((g_1,g_1'), g_2)\in G$ we have $\widehat{\varphi}_{g_2}(g_1,g_1')=(g_2g_1,g_2g_1'),$ is isomorphic to Galois group of function field extension $k(X_1(N^2))\slash k(X_0(N)).$ In particular $k(X_0(N))=k(X_1(N^2))^G.$ \\
Let subgroup $H$ of $G$ be defined as $$H=(\zNz\times \{0\})\rtimes_{\widehat{\varphi}}(\zNz)^{\times},$$ and let $\pi_{1,N},\pi_{2,N},\psi_{N}$ be pullback operators defined by
\begin{align*}
\pi_{1,N}^{*} &\colon k(X_1(N))\xrightarrow{} k(X_1(N^2)),\ \pi_{1,N}((E,R))=(E,[N]R),\\
\pi_{2,N}^{*} &\colon k(X_1(N))\xrightarrow{} k(X_1(N^2)),\ \pi_{2,N}((E,R))=(E\slash \langr{[N]R},R+ \langr{[N]R}),\\
\psi_N^{*} &\colon k(X_0(N))\xrightarrow{} k(X_1(N)),\ \psi_N((E,P))=(E,\langr{P}).
\end{align*}
Functions from the set $\pi_{1,N}^{*}(\psi_N^{*}(k(X_0(N))))$ are invariant under the action of group $G$ and functions from the set $\pi_{2,N}^{*}(\psi_N^{*}(k(X_0(N))))$ are invariant under the action of subgroup $H.$

\begin{proof}
Let $E$ be an elliptic curve over the field $k$ and $P$ a point of order $N$ on that curve. Let $R$ and $R'$ be points of order $N^2$ on curve $E$ and $R$ such that $P=[N]R.$ Pair $(E,R)$ is an enhanced elliptic curve for $\G_1(N^2).$ We are interested in preimages of composition $\pi_{1,N}^{*}\circ\psi_N^{*},$ i.e. in map $(E,R)\mapsto (E,\langr{[N]R}).$ Different $R$ and $R'$ are mapped to the same point if $\langr{[N]R}=\langr{[N]R'},$ so there exists $k\in\N$ such that $[kN]R=[N]R'.$ Because $R'$ is a point of order $N^2$ and $[k]P=[N]R'$ it follows that $\text{gcd}(N,k)=1.$ Altogether, points
\begin{equation*}\label{eqn:pointsForG}
R'=[k]R+\overline{P},\ \text{where}\ \overline{P}\in E[N]\ \text{and}\ k\in\N,\ \text{gcd}(k,N)=1,
\end{equation*}
are mapped by $(E,R)\mapsto (E,\langr{[N]R})$ to the same final point. Number of preimages of this type is $N^2\varphi(N).$\footnote{Throughout the proof $\varphi$ denotes Euler totient function.}

Define $G_1\coloneqq(\zNz)^{2}$ and $G_2\coloneqq(\zNz)^{\times}.$ Let torsion group $E[N]$ be generated by the basis $\langr{P_1,P_2},$ so a point $\overline{P}\in E[N]$ can be expressed as $\overline{P}=[a]P_1+[b]P_2$ for some $a,b\in\zNz.$ Point $R'$ is equal to $[k]R+[a]P_1+[b]P_2.$ We define action of the triple $(a,b,k)\in G_1\rtimes_{\widehat{\varphi}}G_2$ on the point $R,$ i.e. on the set of preimages, with 
\begin{equation}\label{eqb:actionOfG}
(a,b,k)R\mapsto [k]R+[a]P_1+[b]P_2.
\end{equation}

This is a well-defined action, because for two such triples $(a_1,b_1,k_1),(a_2,b_2,k_2),$ we have: 
\begin{align*}
(a_1,b_1,k_1)\circ(a_2,b_2,k_2)R
	&=(a_1,b_1,k_1)([k_2]R+[a_2]P_1+[b_2]P_2) \\
	&=[k_1]([k_2]R+[a_2]P_1+[b_2]P_2)+[a_1]P_1+[b_1]P_2 \\
	&=[k_1k_2]R+[k_1a_2+a_1]P_1+[k_1b_2+b_1]P_2 \\
	&=(a_1+k_1a_2,b_1+k_1b_2,k_1k_2)R.
\end{align*}
Let $G=G_1\rtimes_{\widehat{\varphi}}G_2$ and $\widehat{\varphi}_{k}(a,b)=(ka,kb),\ a,b\in G_1, k\in G_2.$ We have identified functions from $k(X_0(N))$ with their double pullbacks first by $\psi_N$ and then by $\pi_{1,N}.$ More generally, a function field $k(X_0(N))$ was identified with $\pi_{1,N}^{*}(\psi_{N}^{*}(k(X_0(N))).$ Set of preimages of functions in $\pi_{1,N}^{*}(\psi_N^{*}(k(X_0(N))))$ is invariant under the action (\ref{eqb:actionOfG}) of group $G$ which implies $k(X_0(N))=k(X_1(N^2))^G.$

For $H,$ we are interested in the preimages of composition $\pi_{2,N}^{*}\circ\psi_N^{*},$ i.e. in map $(E,R)\mapsto(E\slash\langr{[N]R}, \langr{R+\langr{[N]R}}).$ As before, for the condition $\langr{[N]R}=\langr{[N]R'}$ there exists $\hat{h}\in\N,\ \text{gcd}(\hat{h},N)=1$ such that $R'=[\hat{h}]R+\overline{P},$ for some $\overline{P}\in E[N]$. When $R$ and $R'$ are satisfying this, second condition becomes $\langr{R+\langr{[N]R}}=\langr{R'+\langr{[N]R}},$ i.e. $\langr{R+\langr{P}}=\langr{R'+\langr{P}}.$ Now, there exists $\hat{j},\hat{s}$ such that $[\hat{j}]R-R'=[\hat{s}]P.$ Combining everything together, $[\hat{j}]R-[\hat{h}]R-\overline{P}=[\hat{s}]P,$ and $[\hat{j}-\hat{h}]R=[\hat{s}]P+\overline{P}.$ Right side of this equality is a point of order dividing $N$, so $N|(\hat{j}-\hat{h})$ and there exist $\hat{t}$ such that $\hat{j}-\hat{h}=N\hat{t}.$ Now, $[\hat{t}]P=[\hat{s}]P+\overline{P}$ meaning $\overline{P}\in\langr{P}.$ Altogether, points of the form 
\begin{equation*}\label{eqn:pointsForH}
 R'=[h]R+\overline{P},\ \text{where}\ \overline{P}\in\langr{P}\ \text{and}\ h\in\N,\ \text{gcd}(h,N)=1
\end{equation*}
are mapped by $(E,R)\mapsto(E\slash\langr{[N]R}, \langr{R+\langr{[N]R}})$ to the same final point. Number of preimages of this type is $N\varphi(N).$ The difference here is that we are not working with the whole torsion group $E[N],$ but with subgroup generated by a point $P$ of order $N.$ Using a analogous calculation as for $G,$ functions in $\pi_{2,N}^{*}(\psi_N^{*}(k(X_0(N))))$ are invariant under the action of subgroup ${H=(\zNz\times \{0\})\rtimes_{\widehat{\varphi}}(\zNz)^{\times}.}$
\end{proof}
\end{lemma}

Subgroup $H$ from Lemma \ref{lemma:MainLemma} can be used to define a function field $k'\coloneqq k(X_1(N^2))^H.$ Field $k'$ is an intermediate field $k(X_0(N))\subset k'\subset k(X_1(N^2))$ and a function field for some modular curve, so we can take $k'=k(X(\G')),$ where $\G'$ is a congruence subgroup and $X(\G')\coloneqq\G'\slash\calH.$ All functions from the set $\pi_{2,N}^{*}(\psi_N^{*}(k(X_0(N))))$ are well-defined on the quotient $X(\G')$ due to their invariant property under the action of $H$. From the construction above, $\G'$ is a subset of $\G_0(N)$ and from the calculated number of preimages, index $[\G_0(N):\G']=N.$ The congruence subgroup $\G'$ can be calculated similarly to the congruence subgroup $\tildeG(N)$ from the previous section. 

Using the setup and the proof of Lemma \ref{lemma:MainLemma} and the discussion above, we can prove the following theorem.

\begin{theorem}\label{thm:MainTheorem}
Let $H$ be a group $(\zNz\times \{0\})\rtimes_{\widehat{\varphi}}(\zNz)^{\times}.$ Let $k'$ be a  function field defined as $k'\coloneqq k(X_1(N^2))^H.$ Extension $k'\slash k(X_0(N))$ is not a Galois extension.
\begin{proof}
Let group $G$ and pullbacks $\pi_{1,N},\pi_{2,N},\psi_{N}$ be defined as in Lemma \ref{lemma:MainLemma}. As discussed above, $k'$ is by definition an intermediate field $k(X_0(N))\subset k'\subset k(X_1(N^2))$ and there exists a congruence subgroup $\G'$ such that  $k'=k(X(\G')).$ Working with function fields shown in Figure \ref{fig:groupsX0Connections}, to get radical isogeny formulas on $X_0(N),$ we need to find an $\alpha\in k(X_0(N))$ such that $k(X_0(N))(\sqrt[N]{\alpha})=k(X(\G')).$ Functions from $k(X_0(N))$ are identified with composition of pullbacks $\pi_{1,N}$ and $\psi_N,$ i.e. $\alpha$ should be an element of the field $\pi^{*}_{1,N}(\psi^{*}_N(k(X_0(N)))).$ If such $\alpha$ exists, field extension $k(X(\G'))\slash k(X_0(N))$ should be a cyclic extension of order $N,$ i.e. it should be a Galois extension. This implies that $H,$ a subgroup of index $N,$ should be a normal subgroup of $G.$

Points of type $R'=R+[l]P,\ l\in\N$ are mapped by $(E,R)\mapsto(E\slash\langr{[N]R},R+\langr{[N]R})$ to the same final point. Corresponding congruence subgroup describing preimages of this type was calculated in Proposition \ref{prop:WellDefModularCurve} and it is equal to $$\tildeG(N)=\left\{\sm{a}{b}{c}{d}\in \SL_2(\Z)\colon\modulo{c}{0}{N^2},\ \modulo{a,d}{1}{N}\right\}.$$ The index $[\tildeG(N)\colon\G_1(N^2)]$ is equal to $N.$ This, combined with the calculated number of preimages in the proof of Lemma \ref{lemma:MainLemma}, implies that $\tildeG(N)\subset\G'$ with index equal to $\varphi(N).$ Function $\beta'$ is an element of $k(X(\tildeG(N)))$ by definition and an element of $k(X(\G'))$ by construction. 

If $H$ is a normal subgroup, then for every $g\in G$ and every $h\in H$ there should exist some $h'\in H$ such that $ghg^{-1}=h'.$ Let $g=((g_1,g_2),k_1)\in G$ and $h=((h_1,0),k_2)\in H$. Using $g$ and $h$,
\begin{align*}
ghg^{-1}
&=((g_1,g_2),k_1)((h_1,0),k_2)((g_1,g_2),k_1)^{-1} \\
&=((g_1,g_2),k_1)((h_1,0),k_2)(\widehat{\varphi}_{k_1^{-1}}((g_1,g_2)^{-1}),k_1^{-1}) \\
&=((g_1,g_2)\widehat{\varphi}_{k_1}(h_1,0),k_1k_2)(\widehat{\varphi}_{k_1^{-1}}(-g_1,-g_2),k_1^{-1}) \\
&= ((g_1+k_1h_1,g_2+k_1\cdot 0),k_1k_2)((-k_1^{-1}g_1,-k_1^{-1}g_2),k_1^{-1}) \\
&= ((g_1+k_1h_1-k_2g_1,g_2-k_2g_2),k_2).
\end{align*}    
For this product to be in $H$, $g_2-k_2g_2$ should be equal to $0,$ for every $k_2\in(\zNz)^{\times}$ and every $g_2\in\zNz.$ Let $g_2$ be a generator for $\zNz,$ for example, take $g_2=1.$ Then, for every $k_2\in(\zNz)^{\times}, k_2\neq 1$ we have $k_2g_2=k_2\cdot 1=k_2\neq 1=g_2.$ To conclude, $H$ is not a normal subgroup of $G.$ 
\end{proof}

\end{theorem}

Returning to Example \ref{ex:OpenProblem}, the existence of radical isogeny formulas on $S_0(5)$ depends on finding a parametrization of $S_0(5)$ for which the extension $\Q(\zeta_5)(\beta')\slash\Q(\zeta_5)(\beta)$ is Galois. However, Theorem \ref{thm:MainTheorem} proves that a Galois extension is not possible in a more generalized setting of modular curves. As a direct consequence of that fact, we have the following corollary which is the main result of this article.

\begin{corollary}\label{cor:OpenProblem}
Let $N \geq 5.$ Radical isogeny formulas on $S_0(N)$ are not possible.
\end{corollary}

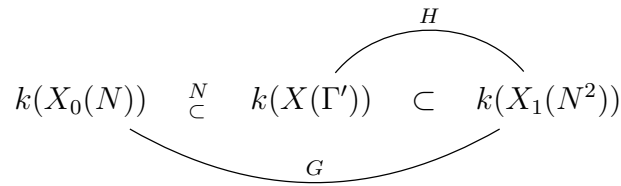
\begin{figure}
\begin{center}
\begin{tikzcd}
      & k(X_0(N)) \arrow[r, phantom, "_\subset^N" description] \arrow[rr, no head, bend right=30, "G"]
      & k(X(\G')) \arrow[r, phantom, "\subset" description] \arrow[r, no head, bend left=50, "H"]
      & k(X_1(N^2))                        
\end{tikzcd}
\caption{Function fields related to groups $G$ and $H$}
\label{fig:groupsX0Connections}
\end{center}
\end{figure}

\bibliographystyle{plain}
\bibliography{refs}
\end{document}